\documentclass[11pt]{article}
\usepackage[margin=1in]{geometry}
\usepackage{amssymb,amsfonts,amsmath,bbm,mathrsfs,stmaryrd}
\usepackage{xcolor}
\usepackage{url}

\usepackage{enumerate}

\usepackage{hyperref}
\hypersetup{colorlinks,
             linkcolor=black!75!red,
             citecolor=blue,
             pdftitle={},
             pdfproducer={pdfLaTeX},
             pdfpagemode=None,
             bookmarksopen=true
             bookmarksnumbered=true}

\usepackage{tikz}
\usetikzlibrary{arrows,calc,decorations.pathreplacing,decorations.markings,intersections,shapes.geometric,through,fit,shapes.symbols,positioning,decorations.pathmorphing}

\usepackage{braket}

\usepackage[amsmath,thmmarks,hyperref]{ntheorem}
\usepackage{cleveref}

\creflabelformat{enumi}{#2(#1)#3}

\crefname{section}{Section}{Sections}
\crefformat{section}{#2Section~#1#3} 
\Crefformat{section}{#2Section~#1#3} 

\crefname{subsection}{\S}{\S\S}
\AtBeginDocument{%
  \crefformat{subsection}{#2\S#1#3}%
  \Crefformat{subsection}{#2\S#1#3}%
}

\crefname{subsubsection}{\S}{\S\S}
\AtBeginDocument{%
  \crefformat{subsubsection}{#2\S#1#3}%
  \Crefformat{subsubsection}{#2\S#1#3}%
}

\theoremstyle{plain}

\newtheorem{lemma}{Lemma}[section]
\newtheorem{proposition}[lemma]{Proposition}
\newtheorem{corollary}[lemma]{Corollary}
\newtheorem{theorem}[lemma]{Theorem}

\newtheorem{question}[lemma]{Question}

\theoremstyle{nonumberplain}

\theoremstyle{plain}
\theorembodyfont{\upshape}
\theoremsymbol{\ensuremath{\blacklozenge}}

\newtheorem{definition}[lemma]{Definition}
\newtheorem{example}[lemma]{Example}
\newtheorem{remark}[lemma]{Remark}
\newtheorem{convention}[lemma]{Convention}
\newtheorem{notation}[lemma]{Notation}

\crefname{definition}{definition}{definitions}
\crefformat{definition}{#2definition~#1#3} 
\Crefformat{definition}{#2Definition~#1#3} 

\crefname{ex}{example}{examples}
\crefformat{example}{#2example~#1#3} 
\Crefformat{example}{#2Example~#1#3} 

\crefname{remark}{remark}{remarks}
\crefformat{remark}{#2remark~#1#3} 
\Crefformat{remark}{#2Remark~#1#3} 

\crefname{convention}{convention}{conventions}
\crefformat{convention}{#2convention~#1#3} 
\Crefformat{convention}{#2Convention~#1#3} 

\crefname{notation}{notation}{notations}
\crefformat{notation}{#2notation~#1#3} 
\Crefformat{notation}{#2Notation~#1#3} 

\crefname{table}{table}{tables}
\crefformat{table}{#2table~#1#3} 
\Crefformat{table}{#2Table~#1#3}

\crefname{lemma}{lemma}{lemmas}
\crefformat{lemma}{#2lemma~#1#3} 
\Crefformat{lemma}{#2Lemma~#1#3} 

\crefname{proposition}{proposition}{propositions}
\crefformat{proposition}{#2proposition~#1#3} 
\Crefformat{proposition}{#2Proposition~#1#3} 

\crefname{corollary}{corollary}{corollaries}
\crefformat{corollary}{#2corollary~#1#3} 
\Crefformat{corollary}{#2Corollary~#1#3} 

\crefname{theorem}{theorem}{theorems}
\crefformat{theorem}{#2theorem~#1#3} 
\Crefformat{theorem}{#2Theorem~#1#3} 

\crefname{enumi}{}{}
\crefformat{enumi}{(#2#1#3)}
\Crefformat{enumi}{(#2#1#3)}

\crefname{assumption}{assumption}{Assumptions}
\crefformat{assumption}{#2assumption~#1#3} 
\Crefformat{assumption}{#2Assumption~#1#3} 

\crefname{equation}{}{}
\crefformat{equation}{(#2#1#3)} 
\Crefformat{equation}{(#2#1#3)}

\numberwithin{equation}{section}

\theoremstyle{nonumberplain}
\theoremsymbol{\ensuremath{\blacksquare}}

\newtheorem{proof}{Proof}
\newcommand\pf[1]{\newtheorem{#1}{Proof of \Cref{#1}}}
\newcommand\pff[3]{\newtheorem{#1}{Proof of \Cref{#2}#3}}

\newcommand\bC{{\mathbb C}}

\newcommand\bQ{{\mathbb Q}}
\newcommand\bR{{\mathbb R}}
\newcommand\bS{{\mathbb S}}

\newcommand\bZ{{\mathbb Z}}

\newcommand\cS{{\mathcal S}}

\newcommand\fa{{\mathfrak a}}

\newcommand\fc{{\mathfrak c}}

\newcommand\fg{{\mathfrak g}}
\newcommand\fh{{\mathfrak h}}

\newcommand\fk{{\mathfrak k}}
\newcommand\fl{{\mathfrak l}}

\newcommand\fp{{\mathfrak p}}

\newcommand\fs{{\mathfrak s}}
\newcommand\ft{{\mathfrak t}}
\newcommand\fu{{\mathfrak u}}

\DeclareMathOperator{\id}{id}

\newcommand{\cat}[1]{\textsc{#1}}

\newcommand{\qedhere}{\mbox{}\hfill\ensuremath{\blacksquare}}

\title{Large sets of generating tuples for Lie groups}
\author{Alexandru Chirvasitu}

\begin{document}

\date{}

\newcommand{\Addresses}{{
  \bigskip
  \footnotesize

  \textsc{Department of Mathematics, University at Buffalo, Buffalo,
    NY 14260-2900, USA}\par\nopagebreak \textit{E-mail address}:
  \texttt{achirvas@buffalo.edu}

}}

\maketitle

\begin{abstract}
  We prove that for a connected, semisimple linear Lie group $G$ the spaces of generating pairs of elements or subgroups are well-behaved in a number of ways: the set of pairs of elements generating a dense subgroup is Zariski-open in the compact case, Euclidean-open in general, and always dense. Similarly, for sufficiently generic circle subgroups $H_i$, $i=1,2$ of $G$, the space of conjugates of $H_i$ that generate a dense subgroup is always Zariski-open and dense. Similar statements hold for pairs of Lie subalgebras of the Lie algebra $Lie(G)$.
\end{abstract}

\noindent {\em Key words: Lie group; Lie algebra; semisimple; discrete subgroup; Zariski topology; dense}

\vspace{.5cm}

\noindent{MSC 2020: 22E46; 22E60; 22C05; 17B20; 20G20; 14L15}

\tableofcontents

\section*{Introduction}

The general theme of the present paper is that of (topologically) generating Lie groups by pairs of elements and/or subgroups. There are a number of results to this effect in the literature:

\begin{itemize}
\item A classical theorem of Auerbach states that compact connected matrix groups are generated by pairs of elements, and in fact by ``most'' such pairs: e.g. the original \cite[Th\`eor\'eme II]{aub3} or \cite[Theorem 6.82]{hm}.
\item This was later generalized by Schreier and Ulam, who show in \cite{su} that for a compact, connected metrizable group $G$, the subset of $G\times G$ consisting of pairs that do {\it not} generate $G$ topologically is of $1^{st}$ Baire category. Or to say it differently, the set of generating pairs is {\it residual}.
\item On the other hand, for {\it non-abelian} compact connected Lie groups one can ask for more: the set of pairs of elements that generate a dense {\it free} subgroup is residual \cite[Corollary 6.87]{hm}.
\end{itemize}

In this same spirit, we consider problems of the following general shape: $G$ will be a connected Lie group (typically linear, etc.), and, taking a cue from the above-mentioned Auerbach theorem, we are interested in pairs of elements $g,h\in G$ that generate ``large'' subgroups thereof. The discussion then revolves around the general principle that the sets of such pairs $(g,h)\in G$ tend to be ``well-behaved''. To make the two quoted phrases more precise:
\begin{itemize}
\item `large' means that the group $\langle g,h\rangle\subset G$ is dense,
\item while `well-behaved' means `open'.
\end{itemize}
This is still ambiguous: each of the two bullet points tacitly assumes a topology, which in turn entails (at least) a binary choice: Zariski or standard (i.e. Euclidean; cf. \Cref{subse:alggps}). As per \Cref{qu:4choices} and the ensuing discussion, the more demanding results of this nature will require that tuples which generate {\it Euclidean}-dense subgroups be {\it Zariski}-open; for that reason, these are the types of statements we focus on throughout the paper.

Concretely, the main results are as follows. First, as far as generating element tuples go, there is

\begin{theorem}[\Cref{th:zarop}]
  Let $G$ be a connected, semisimple, Linear Lie group and $k\ge 2$ a positive integer.

  The space of $k$-tuples that generate a Euclidean-dense subgroup of $G$ is Zariski-open when $G$ is compact and Euclidean-open in general.  
\end{theorem}

The compactness constraint cannot be immediately removed because discrete subgroups will not allow it (\Cref{re:latzardense}), but with that obstacle out of the way we have Zariski-openness in general.

\begin{theorem}[\Cref{th:mostsubalg}]
  Let $G$ be a connected, semisimple, Linear Lie group and $k\ge 2$ a positive integer.

  The space of $k$-tuples of Lie subalgebras $\fg_i\le \fg:=Lie(G)$ whose underlying Lie subgroups of $G$ generate a Euclidean-dense subgroup is Zariski-dense in the $k$-fold Grassmannian power $\mathrm{Gr}(\fg)^k$.
\end{theorem}

Finally, it would be helpful if the above open sets were, in fact, non-empty. Then the Zariski-openness would mean that they are truly ``large'' (full Lebesgue measure, second Baire category, dense in any topology, etc.). In this direction, we have the following result, paraphrased from \Cref{th:strreggen} for brevity. 

\begin{theorem}
  Let $G$ be a connected, semisimple, Linear Lie group, $k\ge 2$ a positive integer, and $H_i\le G$, $i=1,2$ two appropriately generic circle subgroups.

  The space of pairs $(g_1,g_2)\in G^2$ for which the conjugates $Ad_{g_i}H_i$ generate a Euclidean-dense subgroup of $G$ is non-empty, and hence Zariski-open and Euclidean-dense.
\end{theorem}

`Sufficiently generic' will be unwound below: it means `strongly regular', per \Cref{def:reg}.

\subsection*{Acknowledgements}

I am grateful for I. Penkov's advice and pointers to the literature and A. Sikora's helpful comments.

This work is partially supported through NSF grant DMS-2001128.

\section{Preliminaries}\label{se.prel}

Lie algebras are assumed finite-dimensional unless specified otherwise. Following standard practice (e.g. \cite[\S 2.1]{hum}, \cite[\S 9.1]{fh}, \cite[Chapter VI, Definition 2.3]{ser-lie} and so on), {\it simple} Lie algebras are always assumed non-abelian so as to obviate the need to handle the 1-dimensional Lie algebra separately.

\begin{notation}
  The operator $Lie(-)$ does double duty, depending on context: it means either the Lie algebra of a Lie group, or, when applied to a Lie subalgebra $\fl\le \fg$ of a Lie group $G$ (as in $Lie(\fl)$), it means the connected Lie subgroup of $G$ with Lie algebra $\fl$.
\end{notation}

\subsection{Semisimple Lie groups and algebras}\label{subse:ssliealg}

In addition to standard references on complex semisimple Lie algebras (e.g. \cite{hum}, \cite[Part IV]{fh} and so on), we will need some material on {\it real} forms thereof. Good references for this are \cite[Chapters VI and IX]{helg} and \cite[Chapters VI and VII]{knp}. The latter, in particular, gives a treatment that will be convenient below. What follows is by no means meant as an exhaustive introduction, but rather as a very brief review of the terminology and setup.

\begin{definition}\label{def:cartinv}
  For a real semisimple Lie algebra $\fg$ a {\it Cartan involution} is an order-$(\le 2)$ automorphism $\theta:\fg\to \fg$ for which the bilinear form
  \begin{equation*}
    B_{\theta}(x,y):=-\text{(Killing form)}(x,\theta y)
  \end{equation*}
  is positive-definite.

  Any two Cartan involutions are conjugate under the adjoint group $\mathrm{Int}(\fg)$. 
\end{definition}
(See \cite[\S VI.2]{knp}). 

The $1$ and $(-1)$-eigenspaces $\fk$ and $\fp$ of a Cartan involution play an important role in the classification of real {\it simple} Lie algebras \cite[\S VI.10]{knp} (arbitrary {\it semi}simple ones being direct sums of the latter \cite[Theorem 1.54]{knp}). Following common convention (e.g. \cite[\S\S VI.2 and VI.3]{knp}), we refer to $(\fk,\fp)$ as a {\it Cartan decomposition} of $\fg$. In summary:

\begin{itemize}
\item The $1$-eigenspace $\ker(\theta-\id)$ is a Lie subalgebra $\fk\le \fg$, whose corresponding Lie subgroup in $\mathrm{Int}(\fg)$ (or in any linear Lie group with Lie algebra $\fg$) is maximal compact in $\fg$ \cite[Theorem 6.31]{knp}.
\item $\fg$ has $\sigma$-invariant {\it Cartan subalgebras} $\fh\subset \fg$ (i.e. Lie subalgebras whose complexification is Cartan in $\fg_{\bC}:=\fg\otimes_{\bR}\bC$ in the sense of complex Lie algebras \cite[\S 15]{hum}): \cite[Proposition 6.47]{knp}. Being $\sigma$-invariant, these will decompose as direct sums
  \begin{equation*}
    \fh=(\fh\cap \fk)\oplus(\fh\cap\fp),
  \end{equation*}
  where $\fk$ and $\fp$ are the $(\pm 1)$-eigenspaces of $\theta$ respectively. 
\item Among the $\theta$-invariant Cartan subalgebras, of special interest are the {\it maximally compact ones}, for which $\dim(\fh\cap\fk)$ is as large as possible, and the {\it maximally non-compact} ones, for which $\dim(\fh\cap\fp)$ is maximal instead (\cite[discussion preceding Proposition 6.60]{knp}).
\item While Cartan subalgebras are not all conjugate under $\mathrm{Int}(\fg)$,
  \begin{enumerate}[(a)]
  \item they all have the same dimension (because all Cartan subalgebras of the complexification $\fg^{\bC}$ do \cite[Theorem 2.15]{knp});
  \item they fall into finitely many $\mathrm{Int}(\fg)$-conjugacy classes \cite[Proposition 6.64]{knp};
  \item the ($\sigma$-invariant) maximally compact ones are mutually conjugate under the Lie group $K\subset \mathrm{Int}(\fg)$ with Lie algebra $\fk$, as are the maximally non-compact ones \cite[Proposition 6.61]{knp}.
  \end{enumerate}  
\end{itemize}

\begin{remark}\label{re:charind}
  It follows from the same \cite[Theorem 6.31]{knp} (or from the {\it Iwasawa decomposition} of $G$; \cite[Theorem 6.46]{knp}, \cite[\S 4.13, Theorem]{mz} or the original \cite[Theorem 13]{iw}) that a semisimple, connected Lie group with finite center is homeomorphic to $K\times \bR^r$ for a maximal compact subgroup $K$. The number $r$ is the {\it characteristic index} $\mathrm{c}(H)$ of $H$.
\end{remark}

\begin{notation}\label{not:quad}
  For brevity we will, on occasion, package the data of a Cartan involution $\theta$ on $\fg$ with corresponding decomposition $\fg=\fk\oplus \fp$ as a quadruple $(\fg,\theta,\fk,\fp)$. 
\end{notation}

Now consider a Cartan decomposition $(\fg,\theta,\fk,\fp)$ of a real semisimple Lie algebra $\fg$ (lettering as above), and let
\begin{equation*}
  \fh=(\fh\cap \fk)\oplus(\fh\cap\fp) =: \ft\oplus \fa
\end{equation*}
be a maximally compact $\theta$-invariant Cartan subalgebra. This then also provides the system of roots $\Delta$ associated to the {\it complex} Cartan subalgebra $\fh^{\bC}\subset \fg^{\bC}$. As explained in the discussion opening \cite[\S VI.8]{knp}, it is possible to choose a system $\Delta^+$ of positive roots invariant (as a set) under the involution $\theta$: it is enough, for instance, to
\begin{itemize}
\item define an order on
\begin{equation}\label{eq:distbasis}
  (\text{a basis for }i\ft)\cup(\text{a basis for }\fa)
\end{equation}
that places the former elements before the latter;
\item and then define a root to be positive if it takes a positive value on the smallest element of \Cref{eq:distbasis} where it is non-zero.
\end{itemize}
Having fixed this setup (always assumed whenever we refer to positivity, simple roots, etc. in the context of Cartan subalgebras), one approach to the classification of simple Lie algebras is via the {\it Vogan diagram} attached to this data \cite[\S VI.8]{knp}:
\begin{itemize}
\item the Dynkin diagram \cite[\S II.5]{knp} of the complexification $\fg^{\bC}$;
\item with arrows joining the pairs of dots / simple roots that are equal when restricted to the compact piece $\ft$ of $\fh$ (i.e. those pairs of simple roots, if any, that are interchanged by the involution $\theta$);
\item and those dots corresponding to simple roots $\alpha$ that are fixed by $\theta$ (i.e. {\it imaginary} $\alpha$) and whose eigenspace is contained in the complexification $\fp^{\bC}$ (i.e. {\it non-compact} $\alpha$) painted black.
\end{itemize}

Simple real Lie algebras fall into two qualitative classes \cite[Theorem 6.94]{knp}:
\begin{enumerate}[(a)]
\item\label{item:23} complex simple Lie algebras regarded as real algebras by scalar restriction;
\item\label{item:22} non-complex ones. 
\end{enumerate}
The Vogan diagram is connected (in the sense that the underlying Dynkin diagram is) precisely for simple Lie algebras of type \Cref{item:22}: \cite[discussion preceding Theorem 6.96]{knp}. Since {\it complex} simple Lie algebras are more readily classified than real ones, it is usually case \Cref{item:22} that one usually focuses on in this setting. In that context (connected Vogan diagrams) it can be shown \cite[Theorem 6.96]{knp} that at most {\it one} Dynkin-diagram dot need be painted.

In turn, the non-complex simple Lie algebras fall into two distinct classes, typically treated separately:
\begin{enumerate}[(1)]
\item\label{item:24} those whose underlying Vogan-diagram automorphism is trivial, which is equivalent to saying that a maximally compact Cartan subalgebra $\fh\subset \fg$ consists entirely of its compact piece $\ft\subset \fk$;
\item\label{item:25} those with {\it non}-trivial automorphism group.
\end{enumerate}
In both cases, the rank of the compact group $K\subset\mathrm{Int}(\fg)$ with Lie algebra $\fk=\ker(\theta-\id)$ (i.e. the dimension of a maximal torus therein) equals the number of orbits of $\theta$-orbits displayed as part of the Vogan-diagram structure.

To illustrate, an instance of \Cref{item:24}:

\begin{example}
  As described in \cite[\S VI.8, Example 1]{knp}, the type-A$_n$ Dynkin diagram with trivial automorphism and the $p^{th}$ vertex painted corresponds to the Lie algebra $\mathfrak{su}(p,q)$, whose corresponding maximal compact group is the full-rank $S(U(p)\times U(n-p))\subset SL(n,\bC)$, where the $S(-)$ on the left-hand side indicates determinant-1 elements.
\end{example}

As for \Cref{item:25}, we have 

\begin{example}
  As per \cite[\S VI.10]{knp}, the Dynkin diagram A$_{2n}$ equipped with its only non-trivial involution corresponds to the Lie algebra $\mathfrak{sl}(2n+1,\bR)$. The rank of the corresponding maximal compact subgroup $SO(2n+1)$ is indeed $n$, i.e. the number of orbits of the involution.
\end{example}

We will need some information on the structure of a simple real Lie algebra as a module over a maximal torus of the maximal compact group $K$ associated to a Cartan decomposition. Recall the quadruple-style \Cref{not:quad} for Cartan decompositions.

\begin{lemma}\label{le:tormod}
  Let
  \begin{itemize}
  \item $(\fg,\theta,\fk,\fp)$ be a Cartan decomposition of a real simple Lie algebra $\fg$;
  \item $K\subset G:=\mathrm{Int}(\fg)$ the maximal compact subgroup with Lie algebra $\fk$;
  \item $\fh\subset \fg$  a maximally compact Cartan subalgebra, with decomposition
    \begin{equation*}
      \fh=\ft\oplus \fa,\quad \ft\subset \fk,\ \fa\subset \fp,
    \end{equation*}
  \item and $T\subset K$ the maximal torus with Lie algebra $\ft$.
  \end{itemize}
  The adjoint representation of $T$ on $\fg$ decomposes as a direct sum
  \begin{equation*}
    \fg=\fh\oplus V,
  \end{equation*}
  where $\fh$ carries the trivial $T$-action and $V$ has no trivial subrepresentations. Furthermore,
  \begin{enumerate}[(a)]
  \item\label{item:26} If $\fg$ is complex then $V$ is a sum of 2-dimensional irreducible $T$-representations, each appearing with multiplicity 2.
  \item\label{item:27} If $\fg$ is non-complex and its Vogan diagram has trivial automorphism group then $V$ is a sum of 2-dimensional irreducible $T$-representations, each appearing with multiplicity 1.
  \end{enumerate}
\end{lemma}
\begin{proof}
  That we do indeed have a decomposition $\fg=\fh\oplus V$ with $V$ a sum of 2-dimensional irreducible real $T$-representations follows from
  \begin{itemize}
  \item the compactness of $T$, which means that it acts completely reducibly;
  \item the fact that $\fh$ is precisely the centralizer of $\ft$ in $\fg$, this being the general procedure of obtaining a maximally compact Cartan subalgebra of $\fg$ \cite[Proposition 6.60]{knp};
  \item the fact that the only non-trivial real irreducible representations of a torus are 2-dimensional.
  \end{itemize}
  It remains to justify the multiplicity claims. In the complex case \Cref{item:26} simply note that we have the usual root-space decomposition
  \begin{equation*}
    \fg = \fh\oplus\bigoplus_{\alpha\in\Delta}\fg_{\alpha}
  \end{equation*}
  (\cite[top of p.128]{knp}) for the root system $\Delta:=\Delta(\fg,\fh)$ associated to the Cartan subalgebra $\fh\subset \fg$, and as {\it real} $T$-representations $\fg_{\alpha}$ is isomorphic to $\fg^{\beta}$ precisely for $\alpha=\pm\beta$.

  On the other hand, in the trivial-automorphism-Vogan-diagram case \Cref{item:27} we have an analogous decomposition of the complexification $\fg^{\bC}$:
  \begin{equation*}
    \fg^{\bC} = \ft^{\bC}\oplus\bigoplus_{\alpha\in\Delta} \left(\fg^{\bC}\right)_{\alpha}
  \end{equation*}
  for $\Delta:=\Delta(\fg^{\bC},\ft^{\bC})$. Since that decomposition is obtained as the complexification of the real $T$-representation $\fg$, the latter must contain, along with $\ft$, a copy of each 2-dimensional real irreducible $T$-representation $\left(\fg^{\bC}\right)_{\alpha}$ for {\it positive} $\alpha$, so as to avoid the duplication resulting from $\left(\fg^{\bC}\right)_{\alpha} \cong \left(\fg^{\bC}\right)_{-\alpha}$ (as real $T$-representations).
\end{proof}

\subsection{Algebraic structure}\label{subse:alggps}

We need some background on algebraic groups and group-scheme language, as covered, say, in \cite[Chapters AG, I]{bor}. Another good source is \cite[Preliminaries, sections 1 and 2]{ragh}, giving a brief but essentially complete summary of what we need (and more).

On occasion, we will have to regard Lie groups as real algebraic varieties. For every {\it compact} Lie group $G$ there is a complex algebraic group scheme ${\bf G}$ defined over $\bR$ such that
\begin{equation}\label{eq:gtoalg}
  G\cong {\bf G}({\bR}),
\end{equation}
i.e. $G$ consists of the {\it real points} of the scheme ${\bf G}$ \cite[\S 13]{bor}. This convention of denoting Lie groups by plain capital letters (like `$G$') and group schemes by bold-face letters (like `${\bf G}$') will be in place throughout.

The compact/algebraic link discussed above is a mainstay of representation theory and forms the basis of {\it Tannaka reconstruction}, for which many good presentations exist in the literature. These include \cite[Chapter VI, \S VIII and \S IX ]{chv}, \cite[Chapter 9]{rob}, or the very pithy and readable account in \cite{cas-alg} (which conveniently lists yet more sources).

Something of this sort goes through for more general Lie groups, but with some qualifications. First, when considering algebraic structures it will be convenient to specialize to connected Lie groups. Secondly, {\it linear} Lie groups (i.e. those realizable as closed subgroups of $GL_n(\bR)$ for some $n$) are of particular interest. In this context, we recall from \cite[Preliminaries, \S 2]{ragh} that
\begin{itemize}
\item for every connected, linear Lie group $G$ realized as
  \begin{equation}\label{eq:gingl}
    G\le GL_n(\bR)
  \end{equation}
  we can recover $G$ as the connected component ${\bf G}(\bR)_0$, where ${\bf G}$ is the group scheme obtained as the {\it Zariski closure} of \Cref{eq:gingl};
\item and the topology on $G$ obtained by restricting the Zariski topology of $GL_n(\bR)$ does not depend on the realization \Cref{eq:gingl}.
\end{itemize}
For this reason, there is a canonical Zariski topology on such a Lie group $G$, weaker than what we refer to as the {\it usual} (or {\it standard}, or {\it Euclidean}) topology. 

\section{Open sets of generating elements}\label{se:open}

Fix a Lie group $G$. In investigating the nature of the set of generating pairs $(g,h)\in G^2$ (or pairs of generating subgroups) two issues arise:
\begin{itemize}
\item the openness of those sets under various topologies (Zariski, Euclidean);
\item the non-emptiness of those sets.
\end{itemize}
We focus on the first item here, and defer the problem of non-emptiness until the next section. One reason is that openness questions can be addressed independently of whether or not the sets of interest are empty; another is that later, when proving such sets are, in fact, non-empty, the Zariski openness obtained previously will show that they are in fact ``large'': dense in the Euclidean topology, of full Lebesgue measure, etc.

\subsection{Generation by elements}\label{subse:elts}

\begin{convention}\label{cv:ez}
  Unqualified topological notions (closure, etc.) are to be interpreted as Euclidean, and whenever the Zariski topology is intended the text will make it clear.

  For emphasis though, it will also be convenient, on occasion to decorate terms with either an `e' (for Euclidean) or a `z' (for Zariski):
  \begin{itemize}
  \item `${}_z$closed' means `closed in the Zariski topology';
  \item `${}_e$open' means `Euclidean-open';
  \end{itemize}
  and so forth.
\end{convention}

It will be convenient to introduce the following objects.

\begin{notation}
  For a Lie group $G$ and a positive integer $k\ge 2$ we write
  \begin{itemize}
  \item $G^k_{\cat{prop}}$ for the set of tuples $(g_i)\in G^k$ which topologically generate a closed {\it proper} Lie subgroup $\overline{\langle g_i\rangle}$ of $G$;
  \item $G^k_{\cat{prop}>0}$ for the subset of $(g_i)\in G^k_{\cat{prop}}$ for which $\overline{\langle g_i\rangle}$ is non-discrete;
  \item similarly, $G^k_{\cat{prop}=0}$ for the subset of $(g_i)\in G^k_{\cat{prop}}$ for which $\overline{\langle g_i\rangle}$ {\it is} discrete.
  \end{itemize}
  The plain symbols presuppose the standard topology, but we might decorate them with left-hand subscripts as in \Cref{cv:ez} to indicate one of the topologies of interest: ${}_eG^k_{\cat{prop}}$ is just $G^k_{\cat{prop}}$, but there is a Zariski version ${}_zG^k_{\cat{prop}}$.
\end{notation}

One can now phrase a number of questions (four, depending on how which option is selected at each of the two junctions):

\begin{question}\label{qu:4choices}
  Is $
  \begin{cases}
    {}_eG^k_{\cat{prop}}\\
    {}_zG^k_{\cat{prop}}
  \end{cases}
  $
  closed in the $
  \begin{cases}
    \text{Zariski}\\
    \text{Euclidean}
  \end{cases}
  $ topology?
\end{question}

Schematically:
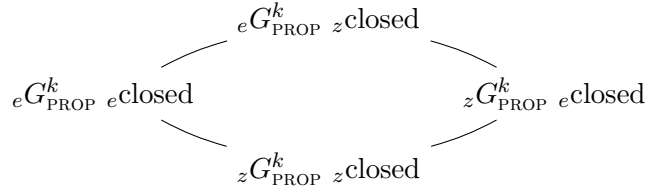
\begin{equation*}
\begin{tikzpicture}[auto,baseline=(current  bounding  box.center)]
\path[anchor=base] 
(0,0) node (l) {${}_eG^k_{\cat{prop}}\ {}_e\text{closed}$}
+(3,1) node (u) {${}_eG^k_{\cat{prop}}\ {}_z\text{closed}$}
+(3,-1) node (d) {${}_zG^k_{\cat{prop}}\ {}_z\text{closed}$}
+(6,0) node (r) {${}_zG^k_{\cat{prop}}\ {}_e\text{closed}$}
;
\draw[-] (l) to[bend left=6] node[pos=.5,auto] {$\scriptstyle $} (u);
\draw[-] (u) to[bend left=6] node[pos=.5,auto] {$\scriptstyle $} (r);
\draw[-] (l) to[bend right=6] node[pos=.5,auto,swap] {$\scriptstyle $} (d);
\draw[-] (d) to[bend right=6] node[pos=.5,auto,swap] {$\scriptstyle $} (r);
\end{tikzpicture}
\end{equation*}

The placement is meant to indicate the relative strength of the topologies. For instance, being Euclidean-dense is {\it harder}, as is being Zariski-closed; this justifies the topmost position for the node labeled `${}_eG^k_{\cat{prop}}\ {}_z\text{closed}$'. Revisiting Auerbach's generation theorem (\cite[Th\`eor\'eme II]{aub3}, \cite[Theorem 6.82]{hm}), one of the main results in this subsection is

\begin{theorem}\label{th:zarop}
  Let $G$ be a connected, semisimple, linear Lie group and $k\ge 2$ a positive integer. Then, the space of tuples $(g_i)_{i=1}^k\in G^k$ that generate $G$ topologically is
  \begin{enumerate}[(a)]
  \item\label{item:15} Zariski-open if $G$ is compact;
  \item\label{item:16} Euclidean-open in general. 
  \end{enumerate}
\end{theorem}

We will need the following simple auxiliary observation; there is certainly nothing new about it, but it is set out for future reference.

\begin{lemma}\label{le:grsisprop}
  Let $Z$ and $W$ be two $\bR$-schemes, with $Z$ projective. Then, the projection map
  \begin{equation*}
    Z(\bR)\times W(\bR)\to W(\bR)
  \end{equation*}
  is Zariski-closed. 
\end{lemma}
\begin{proof}
  Indeed, since $Z$ is projective $Z\times W\to W$ is ${}_z$closed \cite[Theorem 3.12]{har-ag}; we are simply restricting that map to real points.
\end{proof}

A number of further preparatory results, possibly of some independent interest, precede the proof of \Cref{th:zarop}. While most of the discussion goes through unchanged for tuples of arbitrary size $k\ge 2$, the proofs focus on $k=2$ in order to both fix ideas and simplify the notation.

\begin{proposition}\label{pr:genlie}
  Let $k\ge 2$ be a positive integer. For a connected, simple, linear Lie group $G$ the Zariski closure of $G^k_{\cat{prop}>0}$ is contained in $G^k_{\cat{prop}}$.
\end{proposition}
\begin{proof}
  As warned, we work with $k=2$. The claim is that if $g_0,h_0\in G$ {\it do} topologically generate $G$, then some Zariski neighborhood of $(g_0,h_0)\in G^2$ contains no elements of $G^2_{\cat{prop}>0}$.

Consider the subvariety
  \begin{equation}\label{eq:xgrass}
    X:=\{\text{ non-trivial subspaces of $\mathfrak{g}:=Lie(G)$ of less-than-maximal dimension }\}
  \end{equation}
  of the Grassmannian $\mathrm{Gr}(\mathfrak{g})$, and the projection
  \begin{equation}\label{eq:eq:xggr}
    \pi:X\times G\times G\to G\times G
  \end{equation}
  onto the last two factors. The domain $X\times G\times G$ has a Zariski-closed subvariety $Y$ consisting of triples
  \begin{equation}\label{eq:ghv}
    (\text{space V},\ g,\ h)\text{ with }Ad_{g}(V)=V,\ Ad_{h}(V)=V.
  \end{equation}
  Now note that
  \begin{itemize}
  \item $(g_0,h_0)$ does {\it not} belong to the image of $Y$ through $\pi$ because $g$ and $h$ generate $G$ topologically and the latter's Lie algebra is simple;
  \item the image $\pi(Y)$ is Zariski-closed, because $\pi$ is, by \Cref{le:grsisprop} applied with \Cref{eq:xgrass} for the projective scheme $Z$ and ${\bf G}\times{\bf G}$ for $W$ (where ${\bf G}$ is the group scheme attached to $G$, as in \Cref{eq:gtoalg});    
  \item And finally, that closed image $\pi(Y)$ contains all of $G^2_{\cat{prop}>0}$, because for pairs $(g,h)$ therein the Lie algebra of the connected component $\overline{\langle g,h\rangle}_0$ is a space $V$ satisfying \Cref{eq:ghv}.
  \end{itemize}
  This finishes the proof.
\end{proof}

\begin{proposition}\label{pr:gendisc}
  Let $G$ be a connected, non-nilpotent, linear Lie group and $k\ge 2$ a positive integer. $G^k_{\cat{prop}}$ contains the closure of $G^k_{\cat{prop}=0}$
  \begin{enumerate}[(a)]
  \item\label{item:11} in the Zariski topology if $G$ is compact;
  \item\label{item:14} in the standard topology in general. 
  \end{enumerate}
\end{proposition}
\begin{proof}
  Once more, set $k=2$ for simplicity. We have to show that given $g,h\in G$ that generate $G$ topologically, some neighborhood of $(g,h)$ (Zariski in one case, standard in the other) contains no $(g_0,h_0)$ generating a discrete subgroup of $G$. We treat the two cases separately.

  {\bf Case \Cref{item:11}: compact $G$.} Discrete subgroups $\langle g_0,h_0\rangle$ of $G$ are then finite. Jordan's theorem ensures that there is a normal abelian subgroup
  \begin{equation*}
    A\trianglelefteq \langle g_0,h_0\rangle
  \end{equation*}
  of index dominated by some bound $c=c(G)$ depending only on $G$: see e.g. \cite[Theorem 8.29]{ragh}, \cite[Theorems (36.13) and (36.14)]{cr} as well as \cite[Theorem A and the many references therein]{col-j}. Now, because $g$ and $h$ generate $G$, we can find words $w_1$ and $w_2$ in two free group generators and their inverses such that
  \begin{equation}\label{eq:wnc}
    w_1(g,h)^{C!}\text{ and }w_2(g,h)^{C!}\text{ do not commute.}
  \end{equation}
  This is a Zariski-open condition, thus defining a Zariski neighborhood $U$ of $(g,h)$ in $G^2$. Since $(g_0,h_0)$ clearly does not meet the condition, the pairs $(g_0,h_0)$ considered here will lie outside $U$.

  {\bf Case \Cref{item:14}: non-compact $G$.} A result of Zassenhaus (\cite[Theorem 8.16]{ragh}) states that for all $G$, compact or not, there is some neighborhood
  \begin{equation*}
    1\in U\subset G
  \end{equation*}
  such that for any discrete subgroup $\Gamma\subset G$ the intersection $\Gamma\cap U$ is contained in some connected nilpotent Lie subgroup of $G$. In fact, the proof (\cite[\S 8.21]{ragh}) shows more: one can ensure that $\Gamma\cap U$ is contained in the image through the exponential map
  \begin{equation*}
    \exp:\mathfrak{g}:=Lie(G)\to G
  \end{equation*}
  of the intersection $\mathfrak{n}\cap V$, where $V$ and $\mathfrak{n}$ are a small neighborhood of $0$ and a nilpotent Lie subalgebra of $\mathfrak{g}$ respectively.

  By our non-nilpotence assumption, there is an upper bound $<\dim \mathfrak{g}$ for the dimension of a nilpotent Lie subalgebra of $\mathfrak{g}$. Having fixed a metric $d$ topologizing $G$, it follows that if $U$ and $V$ as above are sufficiently small then
  \begin{equation*}
    \sup_{x\in U}d(x,\exp\left(\mathfrak{n}\cap V\right)) \le \sup_{x\in U}d(x,\Gamma\cap U)
  \end{equation*}
  is bounded below by some (strictly) positive constant $\varepsilon>0$ as $\mathfrak{n}$ ranges over the nilpotent Lie subalgebras of $\mathfrak{g}$. On the other hand though, $g,h\in G$ topologically generate $G$, and hence one can find finitely many products of $g^{\pm 1}$ and $h^{\pm 1}$ forming an $\varepsilon$-net in $U$ (i.e. a set within a distance $<\varepsilon$ of any point in $U$). This holds of any $(g_0,h_0)$ sufficiently close to $(g,h)$, so no such pair can generate a discrete subgroup $\Gamma$.
\end{proof}

\begin{remark}\label{re:latzardense}
  One cannot hope to have {\it Zariski} openness in \Cref{pr:gendisc} \Cref{item:14}, for arbitrary (non-compact) semisimple Lie groups. Consider, for instance, the case $G=SL_2(\bR)$ and denote $\Gamma:=SL_2(\bZ)$. The subgroup
  \begin{equation*}
    \Gamma^2 \subset G^2
  \end{equation*}
  is a {\it lattice} \cite[Definition 1.8]{ragh}, i.e. a discrete subgroup such that the pushforward of the Haar measure of $G^2$ to $G^2/\Gamma^2$ is finite. Because $G^2$ is semisimple without compact factors, it is a classical result that its lattices are Zariski-dense (e.g. \cite[Corollary 5.16]{ragh}). This means that no Zariski neighborhood of a topologically-generating pair $(g,h)\in G^2$ can avoid $\Gamma^2$.
\end{remark}

There is a version of \Cref{pr:genlie} valid for {\it semi}simple Lie groups, but the statement is a little involved. Let $\mathfrak{g}:=Lie(G)$ (the Lie algebra of $G$), decompose it as a direct sum
\begin{equation}\label{eq:liealgdec}
  \mathfrak{g}\cong \fs_1\oplus \cdots\oplus \fs_r
\end{equation}
of simple ideals (\cite[\S VI.2, Corollary 3]{ser-lie}), and denote by $\pi:\mathfrak{g}\to \fs_i$ the projections. The Lie algebra $\mathfrak{h}:=Lie(H)$ of a non-discrete, proper closed subgroup $H\le G$ might
\begin{enumerate}[(1)]
\item\label{item:17} map to a proper, positive-dimensional Lie subalgebra of $\fs_i$ through some $\pi_i$, in which case we could shift focus to $\fs_i$ and bring the discussion under the scope of \Cref{pr:genlie};
\item\label{item:18} {\it surject} onto every $\fs_i$ through the corresponding $\pi_i:\mathfrak{g}\to \fs_i$ (while still being proper: $\mathfrak{h}\ne \mathfrak{g}$);
\item\label{item:19} surject through {\it some} $\pi_i$ and be annihilated by others.  
\end{enumerate}
In this latter case the group $H\le G$ will map onto a discrete subgroup of some simple quotient of $G$, forcing us to handle the situation as in \Cref{pr:gendisc}, bifurcating over whether or not we are in the compact case (see also \Cref{re:latzardense}). We will see, however, that case \Cref{item:18} is better behaved. We first need the following observation.

\begin{lemma}\label{le:surjevery}
  Let $\fg$ be a semisimple real Lie algebra with its decomposition \Cref{eq:liealgdec} into simple summands and consider a proper Lie subalgebra $\fh\lneq \fg$ which surjects onto every simple factor $\fs_i$.

  We can then find two simple components $\fs_j\cong \fs_k$ such that the surjection $\fg\to \fs_j\oplus \fs_k$ maps $\fh$ onto the graph of an isomorphism $\fs_j\cong \fs_k$.
\end{lemma}

The following concept will recur a number of times.

\begin{definition}\label{def:homog}
  A semisimple Lie algebra is {\it homogeneous} if its simple summands are mutually isomorphic. A connected semisimple Lie group is homogeneous if its Lie algebra is. 
\end{definition}

\pf{le:surjevery}
\begin{le:surjevery}
  We may as well assume that $\fg$ is homogeneous; indeed, if
  \begin{equation*}
    \fg=\bigoplus_{\ell} \fc_{\ell}
  \end{equation*}
  is the decomposition into homogeneous summands and $\fh$ surjects onto each $\fc_{\ell}$, then
  \begin{itemize}
  \item it must contain a copy of each $\fc_{\ell}$ because surjections onto semisimple Lie algebras split \cite[\S VI.4, Theorem 4.1]{ser-lie};
  \item hence it must coincide with $\fg$ because the $\fc_{\ell}$ have no common simple summands.
  \end{itemize}
  The standing assumption will thus be that $\fg$ itself is homogeneous to begin with, i.e. all $\fs_i$ in \Cref{eq:liealgdec} are isomorphic to a common $\fs$. There will then be a minimal $2\le \ell\le r$ such that
  \begin{itemize}
  \item $\fh$ surjects onto the sum of any $\ell-1$ factors of
    \begin{equation*}
      \fg\cong \bigoplus_{i=1}^r \fs
    \end{equation*}
  \item but it does {\it not} surject onto some factor $\fg_{small}\cong \fs^{\oplus \ell}$. 
  \end{itemize}
  The image $\fh\twoheadrightarrow \fh_{small}$ through the surjection $\fg\twoheadrightarrow \fg_{small}$ surjects onto both displayed summands of 
  \begin{equation}\label{eq:l1l}
    \fg_{small} \cong \fs^{\oplus (\ell-1)}\oplus \fs
  \end{equation}
  and hence must be isomorphic to $\fs^{\oplus(\ell-1)}$. The surjection onto the first summand in \Cref{eq:l1l} must thus be an isomorphism, and hence permute the simple summands of $\fh_{small}\cong \fs^{\oplus(\ell-1)}$ (because those summands are precisely the simple ideals of $\fh_{small}$ \cite[\S VI.2, Corollary 3]{ser-lie}).

  On the other hand, the surjection of $\fh_{small}$ onto the {\it right-hand} summand $\fs$ in \Cref{eq:l1l} must restrict to an isomorphism on one of the simple summands (call it $\ft\cong \fs$) of $\fh_{small}\cong \fs^{\oplus (\ell-1)}$ and annihilate the others. In conclusion, the map of $\fh_{small}$ to
  \begin{equation*}
    (\text{image of $\ft$})\oplus(\text{the trailing $\fs$ in \Cref{eq:l1l}}) 
  \end{equation*}
  identifies $\ft$ with the graph of an isomorphism and annihilates the other simple components of $\fh_{small}$ (which in turn was the image of $\fh$ in $\fg_{small}$). This concludes the proof of the claim.
\end{le:surjevery}

Note, incidentally, the following Helly-type consequence. This is in reference to Helly's theorem, stating that a family of convex subsets of $\bR^d$ any $d+1$ of which intersect must have non-empty intersection \cite[p.102]{dgk-helly}. The underlying ideas have broad combinatorial reach past geometry / convexity (e.g. \cite[\S 11]{bol-bk}), and it is manifest that \Cref{cor:hly} fits in this general circle of ideas.

\begin{corollary}\label{cor:hly}
  Let $\fh\le \fg$ be a Lie subalgebra of a semisimple Lie algebra. If $\fh$ surjects onto the direct sum of any two simple summands of $\fg$, then it must coincide with $\fg$. \qedhere
\end{corollary}

We record some terminology for later reuse.

\begin{definition}\label{def:thn}
  Let $\fg$ be a semisimple Lie algebra with its decomposition \Cref{eq:liealgdec} into simple summands and $\fh\le \fg$ a Lie subalgebra. For a family $\cS:=\{\fs_i\}_{i\in I}$, $I\subseteq 1..r$ of summands we say that $\fh$ is {\it $\cS$-thin} if it surjects onto each individual $\fs_i$, $i\in I$ but maps onto a proper subalgebra of $\bigoplus_{i\in I}\fs_i$.
\end{definition}

\begin{proposition}\label{pr:surjevery}
  Let $G$ be a connected, semisimple, linear Lie group $G$, $k\ge 2$ a positive integer, and denote by $G^k_{\bullet}\subset G^k$ the set of tuples $(g_i)\in G^k$ that
  \begin{itemize}
  \item topologically generate a non-discrete, proper, closed subgroup $H\le G$,
  \item whose Lie algebra $\mathfrak{h}:=Lie(H)$ surjects onto every simple quotient of the Lie algebra $\mathfrak{g}:=Lie(G)$.
  \end{itemize}
  Then, the Zariski closure of $G^k_{\bullet}$ is contained in $G^k_{\cat{prop}}$. 
\end{proposition}
\begin{proof}
  We once more set $k=2$, and fix a direct-sum decomposition \Cref{eq:liealgdec}. The assumption that $\mathfrak{h}$ is proper and surjects onto each simple component $\fs_i$ of that decomposition implies via \Cref{le:surjevery} that there are two isomorphic components $\fs_{j}\cong \fs_k$ such that $\fh$ is $(\fs_j,\fs_k)$-thin in the sense of \Cref{def:thn}. Note, crucially, that thin Lie subalgebras of $\fg$ cannot be ideals, as the ideals are precisely the sums of $\fs_i$s \cite[\S VI.2, Corollary 3]{ser-lie}.
  
  According to \Cref{le:diagalg} the space of $(\fs_j,\fs_k)$-thin Lie subalgebras of $\fg$ is Zariski-closed in the Grassmannian $\mathrm{Gr}(\fg)$, so we can replicate the proof of \Cref{pr:genlie}:
  \begin{itemize}
  \item the thin Lie subalgebras $\fh\subset \fg$ constitute the real points of a projective subscheme $Z$ of the Grassmannian of $\fg_{\bC}$;
  \item hence we can apply \Cref{le:grsisprop} to $Z$ and $W:={\bf G}\times{\bf G}$ to conclude that the set of pairs of elements that lave a thin Lie subalgebra invariant is ${}_z$closed;
  \item and it cannot contain a pair that generates topologically, because as noted, thin Lie subalgebras cannot be ideals.
  \end{itemize}
  We have thus reached the conclusion.
\end{proof}

\begin{lemma}\label{le:diagalg}
  Let $\fs$ be a simple real Lie algebra. The space of proper Lie subalgebras of $\fs\times \fs$ that project onto each factor is ${}_z$closed in the Grassmannian $\mathrm{Gr}(\fs\times \fs)$.
\end{lemma}
\begin{proof}
  Let $\fh\le \fs\times\fs$ be a Lie subalgebra as in the statement. It must then contain an isomorphic copy of $\fs$ identifiable, as a subalgebra of $\fs\times\fs$, with the graph
  \begin{equation*}
    \Gamma_{\alpha}:=\{(x,\alpha(x))\ |\ x\in \fs\}\le \fs\times \fs
  \end{equation*}
  of an automorphism $\alpha:\fs\cong \fs$. The automorphism group of the complexification $\fs^{\bC}:=\fs\otimes_{\bR}\bC$ has the inner automorphisms as a finite-index subgroup (\cite[\S 16.5]{hum}), so
  \begin{equation*}
    \alpha^{\bC}:\fs^{\bC}\to \fs^{\bC}
  \end{equation*}  
  is of the form $\alpha_i\circ Ad_g$ for one of finitely many automorphisms $\alpha_i$ (with the list of possible $\alpha_i$ depending only on the isomorphism class of $\fs$) and some element $g$ of the adjoint group of $\fs^{\bC}$. In particular, since $Ad_g$ will certainly leave some element $0\ne s\in \fs^{\bC}$ invariant, the graph $\Gamma_{\alpha^{\bC}}$ must contain a non-zero element of the form  
  \begin{equation}\label{eq:specpair}
    (s,\alpha_i(s))\text{ for one of the finitely many }\alpha_i.
  \end{equation}
  All in all, the Lie algebra $\fh$ is such that
  \begin{itemize}
  \item when complexified it is proper,
  \item and contains lines through elements of the form \Cref{eq:specpair}. 
  \end{itemize}
  It is now clear that this constitutes a Zariski-closed condition.
\end{proof}

\begin{remark}
  Simplicity (or more generally, semisimplicity) is essential in \Cref{le:diagalg}: for 1-dimensional $\fs$, for instance, the space of 1-dimensional subalgebras of $\fs\times \fs$ that surject onto both factors can be identified with the complement of two points in $\mathrm{Gr}(1,\fs\times\fs)\cong \bS^1$.
\end{remark}

\pf{th:zarop}
\begin{th:zarop}
  We focus mainly on the compact case (and $k=2$), indicating only briefly how the proof can be modified for part \Cref{item:16}.
  
  {\bf \Cref{item:15}: compact $G$.} Sine $G$ surjects with finite central kernel onto a product $S_1\times\cdots \times S_r$ of simple, center-less compact connected Lie groups (\cite[Corollary 9.20]{hm}), we may as well assume
  \begin{equation}\label{eq:gdec}
    G\cong S_1\times \cdots\times S_r
  \end{equation}
  to begin with, with simple center-less $S_i$. In general, for objects pertaining to $G$ (elements, subgroups, etc.) we denote their components in $S_i$ by left-hand subscripts; to illustrate, ${}_ig$ is the component of $g$ in $S_i$. Fix a pair $(g_0,h_0)\in G^2$ that generates $G$ (topologically, which phrase we will often omit for brevity), and consider pairs of elements $g,h\in G$ that do {\it not}. We have to argue that some Zariski-open neighborhood of $(g_0,h_0)$ contains no such $(g,h)$.
  
  Denote
  \begin{equation*}
    H:=\overline{\langle g,h\rangle},\quad \fl:=Lie(H). 
  \end{equation*} 
  Now, $g,h$ might fail to generate for a number of reasons, as in the discussion preceding \Cref{pr:surjevery}:
  \begin{enumerate}[(1)]
  \item\label{item:12} A component ${}_i\fl\le Lie(S_i)$ of the Lie algebra $\fl$ might be proper but positive-dimensional,
  \item\label{item:20} all ${}_i\fl$ are full, though $\fl$ as a whole is proper in $\fg$.
  \item\label{item:13} or the ${}_i\fl$ might all be either full or trivial, with at least one trivial.
  \end{enumerate}
  We consider these in turn.
 
  {\bf Case \Cref{item:12}: For some $i$ we have $0<\dim {}_i \fl< \dim S_i$.} It follows from \Cref{pr:genlie} that the neighborhood of $(g,h)$ obtained as the pullback through $G^2\to S_i^2$ of some Zariski-open neighborhood of $({}_ig,{}_ih)\in S_i^2$ cannot contain any non-generating $(g_0,h_0)$ conforming to the present Case \Cref{item:12}.  

  {\bf Case \Cref{item:20}: $\dim {}_i \fl = \dim S_i$ for all $i$.} Here we can appeal to \Cref{pr:surjevery} instead (of \Cref{pr:genlie}).
  
  {\bf Case \Cref{item:13}: Some component ${}_i\fl$ is trivial.} Or in other words, the subgroup of $S_i$ generated by the $i^{th}$ components ${}_ig_0$ and ${}_ih_0$ is finite. The conclusion follows as in Case \Cref{item:12}, this time appealing to \Cref{pr:gendisc} \Cref{item:11} in place of \Cref{pr:genlie}.

  {\bf \Cref{item:16}: arbitrary $G$.} We can proceed essentially as before, first assuming a decomposition \Cref{eq:gdec} upon modding out a central discrete subgroup (e.g. by passing to the {\it adjoint quotient} of $G$ \cite[Chapter II \S 5]{helg}) and then, when dealing with discrete subgroups, making use of part \Cref{item:14} of \Cref{pr:gendisc} rather than part \Cref{item:11}.
\end{th:zarop}

\begin{remark}\label{re:zarifconn}
  The only reason we had to downgrade the Zariski topology to the Euclidean one in \Cref{th:zarop} \Cref{item:16} is that we had to handle discrete subgroups, which in the non-compact case are harder to control (see \Cref{re:latzardense}). If we work instead only with the pairs $(g,h)\in G^2$ that generate {\it connected} (or at least non-discrete) Lie subgroups, the subset of generating pairs will be open in the relative Zariski topology for arbitrary (possibly non-compact) $G$.
\end{remark}

The openness statement in \Cref{th:zarop} is characteristic of semisimplicity, in the sense that it fails (even for the finer, standard topology) for {\it every} non-semisimple compact, connected Lie group $G$.

\begin{example}\label{ex:t23}
  Indeed, such groups are, up to quotienting out a finite central subgroup, of the form
  \begin{equation*}
    G\cong T\times S_1\times\cdots\times S_r
  \end{equation*}
  for some (non-trivial) torus $T$ and simple compact connected $S_i$ (\cite[Corollary 6.16, Theorem 6.18]{hm}). Then, the elements of $G$ whose $T$-components are torsion form a dense subset. Two such elements $g,h\in G$ cannot generate $G$ topologically, because their images through $G\to T$ do not generate the latter.
\end{example}

\subsection{Generation by connected subgroups}\label{subse:conn}

The point of interest here will be extent to which a Lie group $G$ (as before: semisimple, connected, linear) can be generated (topologically) by two closed connected Lie subgroups, typically of fixed isomorphism class. One might hope for results in the spirit of \Cref{th:zarop}, to the effect that the space of pairs of circles that generate is well-behaved. It is not immediate how to make sense of this though: having fixed $G$, the circle subgroups $\bS^1\le G$ (say) do not constitute a well-behaved space. This is because a line
\begin{equation*}
  \{(cx_1,\cdots,cx_n)\ |\ c\in \bR\}\subset \bR^n\cong Lie(\bR^n/\bZ^n),\ n\ge 2
\end{equation*}
in the Lie algebra of a torus constitutes the Lie algebra of a circle subgroup precisely when the ratios between the coordinates $x_i$ are rational, so the space of such vectors $(x_i)$ is, roughly speaking, a product of copies of $\bQ$.

On the other hand, the Lie subalgebras of $\fg:=Lie(G)$ of fixed dimension form a closed algebraic subvariety of the Grassmannian. This suggests that we parametrize connected subgroups of $G$ by their Lie algebras, and recast or openness/density results as statements about (pairs of) Lie subalgebras of $\fg$. One example:

\begin{theorem}\label{th:mostsubalg}
  Let $G$ be a connected, semisimple, linear Lie group and $k\ge 2$ a positive integer. The space of tuples $(\fg_i)_{i=1}^k$ of Lie subalgebras
  \begin{equation*}
    \fg_i\le \fg:=Lie(G)
  \end{equation*}
  whose underlying Lie subgroups $G_i\le G$ generate the latter topologically is Zariski-open in the Cartesian power $\mathrm{Gr}(\fg)^k$ of the Grassmannian of $\fg$.
\end{theorem}
\begin{proof}
  We can replicate the proof of \Cref{th:zarop}, via \Cref{re:zarifconn}: we are dealing here only with connected Lie groups.
\end{proof}

As for subgroups of $G$ (rather than Lie subalgebras of $\fg$), the approach we take is to fix $H_i\le G$, $i=1..k$ beforehand and parametrize tuples of conjugates
\begin{equation*}
  Ad_{g_i}H_i:= g_i H_ig_i^{-1},\ i=1..k
\end{equation*}
by tuples of elements $g_i\in G$. In these terms, we have the following consequence of \Cref{th:mostsubalg}.

\begin{corollary}\label{cor:twosbgps}
  Let $G$ be a connected, semisimple, linear Lie group and consider $k\ge 2$ connected Lie subgroups $H_i\le G$. The subspace
\begin{equation}\label{eq:genpairs}
  \{(g_i)\in G^k\ |\ Ad_{g_i}H_i\text{ generate $G$ topologically}\}\subseteq G^k
\end{equation}
is Zariski-open.
\end{corollary}
\begin{proof}
  Denoting $\fh_i:=Lie(H_i)$, consider the (algebraic-variety) morphism
  \begin{equation*}
    G^k\ni (g_i)_{i=1}^k\mapsto (Ad_{g_i}\fh_i)_{i=1}^k\in \mathrm{Gr}(\fg)^k. 
  \end{equation*}
  The space \Cref{eq:genpairs} is the preimage through this map of the Zariski-open space of pairs of Lie algebras from \Cref{th:mostsubalg}, hence the conclusion.
\end{proof}

Because we will have to revisit this theme (of $G$ being generated by conjugates of subgroups) a number of times in the sequel, we fix some language.

\begin{definition}\label{def:gengen}
  Let $G$ be a linear, semisimple, connected Lie group and $\{G_i\}_{i\in I}$ a finite family of Lie subgroups thereof. We say that $G_i$ {\it generate $G$ generically} if the subset
  \begin{equation*}
    \{(g_i)_i\in G^I\ |\ Ad_{g_i}G_i\text{ generate $G$ topologically}\} \subseteq G^I
  \end{equation*}
  contains an open dense subset. We might specify the operative topology on either the word `topologically' or on $G^I$ by `e' or `z' lower-left subscripts, as before. So for instance we might say that
  \begin{itemize}
  \item $\{G_i\}$ ${}_e$generate ${}_z$generically (the default, if no decorations are used);
  \item or that they ${}_z$generate ${}_z$generically, etc. (four possible cases).
  \end{itemize}
  The same language applies to Lie subalgebras $\fg_i$ of $\fg:=Lie(G)$ by dispatching back to Lie groups: we say $\{\fg_i\}$ generate generically if their associated Lie subgroups $G_i\le G$ do.
\end{definition}

\section{Density}\label{se:dense}

We now turn to proving that the various sets of generating pairs discussed above are (under reasonable circumstances) non-empty and hence ``large'' by the results obtained in \Cref{se:open}.

We begin with the compact-group case, recovering a version of the aforementioned theorem of Auerbach (\cite[Theorem 6.82]{hm}).

The statement requires the following piece of terminology.

\begin{definition}\label{def:reg}
  Let $T\subseteq G$ be a maximal torus in a connected Lie group, and $S\subset T$ a subtorus. $S$ is
  \begin{enumerate}[(a)]
  \item\label{item:3} {\it regular (in $G$)} if $T$ is the only maximal torus containing it.
  \item\label{item:4} {\it strongly regular (in $G$)} if the restriction
    \begin{equation*}
      \cat{rep}_T\to \cat{rep}_S
    \end{equation*}
    induces a fully faithful functor on the full subcategory generated by the adjoint representation of $T$ on $Lie(G)$.
  \end{enumerate}
  Similarly, elements of $T$ or $Lie(T)$ are (strongly) regular if the trivial connected component of the Lie group they generate is.
\end{definition}

\begin{remark}
  Item \Cref{item:3} of \Cref{def:reg} is compatible with the literature; see e.g. \cite[Corollary 4.35]{ad-lie}, \cite[\S D.1]{fh}, \cite[\S III.3]{helg}, etc..
  
  On the other hand, part \Cref{item:4} requires more: it says that regarded as an $S$-representation, $Lie(G)$ has the same number of irreducible summands as it does over $T$, and the same number of isomorphism classes among those irreducible summands. Strong regularity is (in the Lie version) precisely the condition required of $X$ in \cite[Proposition 6.78, (i)]{hm}. In familiar Lie-theoretic terminology, assuming $G$ is compact:
  \begin{itemize}
  \item regularity for an element $X\in Lie(T)$ means no roots vanish on $X$;
  \item whereas {\it strong} regularity means no roots vanish and no two distinct roots are equal on $X$.
  \end{itemize}
\end{remark}

Note first that for the purpose of generating a group with strongly regular tori, we can move freely between Lie groups with the same Lie algebra:

\begin{lemma}\label{le:samelalg}
  Fix a positive integer $k$. Whether or not $k$ strongly regular tori generically of fixed dimensions $d_i$, $1\le i\le k$ generate a connected, semisimple linear Lie group $G$ depends only on the Lie algebra of $G$.
\end{lemma}
\begin{proof}
  Let $\fg$ be the common Lie algebra of a family of Lie groups as in the statement. All such groups surject onto the adjoint group $\mathrm{Int}(\fg)$ with finite central kernel, and it remains to observe that
  \begin{itemize}
  \item (strongly) regular tori lift and descend through such surjections while retaining their dimension,
  \item and given such a finite-central-kernel surjection $G\to G_{small}$, the circles $H_i\subset G$ generate topologically if and only if their images in $G_{small}$ do. 
  \end{itemize}
  This concludes the argument.
\end{proof}

This then allows us to focus on simple Lie groups.

\begin{proposition}\label{pr:simpenough}
  Let $G$ be a semisimple, connected, linear Lie group and $k$ a positive integer. If every simple quotient of $G$ is generically generated by $k$ strongly regular circles then the same is true of $G$.
\end{proposition}
\begin{proof}
  According to \Cref{le:samelalg} it is enough to assume that $G$ is the adjoint group $\mathrm{Int}(\fg)$ of its Lie algebra, in which case it decomposes as a product of simple components. Generic generation by strongly regular tori passes from two groups to their product, because if $H_{i+}$ and $H_{i-}$ topologically generate $G_+$ and $G_-$ respectively (for $1\le i\le k$) then similarly, $H_{i+}\times H_{i-}$ generate $G_+\times G_-$ topologically.
\end{proof}

We now turn to the main result of this section.

\begin{theorem}\label{th:strreggen}
  Let $G$ be a connected, semisimple, linear Lie group. For any two strongly regular circle subgroups $H_i\subset G$, $i=1,2$ the set \Cref{eq:genpairs} is ${}_z$open and ${}_e$dense.
\end{theorem}

According to \Cref{pr:simpenough} we can restrict attention to simple $G$, which case in turn breaks up into three sub-cases based on the Vogan-diagram classification summarized in \Cref{subse:ssliealg}:
\begin{itemize}
\item complex simple $\fg:=Lie(G)$ regarded as a real Lie algebra;
\item non-complex $\fg$ with trivial-automorphism Vogan diagram;
\item non-complex $\fg$ whose Vogan diagram has non-trivial involution. 
\end{itemize}

Some of those cases will be easier to handle than others.

\pff{th:strreggen-vogtriv}{th:strreggen}{: Vogan diagrams with trivial involution.}
\begin{th:strreggen-vogtriv}
  It suffices to prove the set non-empty: the density will then follow from its Zariski openness (\Cref{cor:twosbgps}). An additional simplification will have us work with subsets of $G$ rather than $G^2$: the automorphism
  \begin{equation*}
    G^2 \ni (g,h)\mapsto (g,g^{-1}h)\in G^2
  \end{equation*}
  maps \Cref{eq:genpairs} onto $G\times X$, where
  \begin{equation*}
    X:=\{g\in G\ |\ H_1\text{ and $Ad_g H_2$ generate $G$ topologically}\}. 
  \end{equation*}
  It will thus be enough to argue that $X$ (and hence $X\subseteq G$) is non-empty.

  Let $\ft\subset \fg$ be a maximally compact Cartan subalgebra, as discussed in \Cref{subse:ssliealg}; we may as well assume that $\ft$ contains the Lie algebra $\fh_1:=Lie(H_1)$. The hypothesis on Vogan diagrams having trivial automorphism groups ensures that in fact $\ft$ is compact, i.e. the Lie algebra of a torus $T\subset G$ of dimension $\mathrm{rank}(\fg^{\bC})$.

  Furthermore, per \Cref{le:tormod}\Cref{item:27}, the supplement $\fg\ominus \ft$ in the adjoint representation of $T$ on $\fg$ decomposes as a sum of {\it multiplicity-1} irreducible 2-dimensional representations $V_{\alpha}$, one for each positive root $\alpha\in \Delta^+$ attached to the Cartan subalgebra $\ft^{\bC}\subset \fg^{\bC}$. It will thus be enough to observe that some conjugate
  \begin{equation*}
    Ad_g(\fh_2:=Lie(H_2))
  \end{equation*}
  has non-trivial component in each $V_{\alpha}$: the Lie algebra of Lie group topologically generated by $H_1$ and
  \begin{equation*}
    \overline{\langle H_1, Ad_g(H_2)\rangle} = \overline{\langle H_1, \mathrm{exp}\left(Ad_g(\fh_2)\right)\rangle} 
  \end{equation*}
  will then contain all $\mathrm{exp}(V_{\alpha})$, which in turn generate $\fg$.

  Now, for each $\alpha$ the set of those $g\in G$ for which $Ad_g(\fh_2)$ has non-trivial $V_{\alpha}$ component in the decomposition
  \begin{equation*}
    \fg = \ft\oplus {\bigoplus}_{\alpha\in \Delta^+(\fg^{\bC},\ft^{\bC})}V_{\alpha}
  \end{equation*}
  is ${}_z$open, so the intersection of all of these sets will be non-empty if each individual one is; it thus suffices to focus on a single $\alpha$.

  To conclude as indicated,
  \begin{itemize}
  \item assume we initially have $\fh_2\subset \ft$ (possible after performing a conjugation);
  \item choose some basis $\{e\}$ for $\fh_1\cong \bR$;
  \item and a non-zero element $e_{\alpha}\in V_{\alpha}\cong \bR^2\cong \bC$.
  \end{itemize}
  Now note that
  \begin{equation*}
    ad_{e_{\alpha}}e = [e_{\alpha},e] = \alpha(e) e_{\alpha}\in V_{\alpha}
  \end{equation*}
  is a (complex) scalar multiple of $e_{\alpha}$. Since
  \begin{equation}\label{eq:diffad}
    ad_{e_{\alpha}}e = \lim_{t\to 0}\frac 1t\left(Ad_{g_t}e-e\right)
  \end{equation}
  for some one-parameter subgroup $t\mapsto g_t\in G$, $Ad_g e$ cannot all have vanishing $V_{\alpha}$-component.
\end{th:strreggen-vogtriv}

This portion of \Cref{th:strreggen} applies in particular to compact semisimple Lie groups, since the simple compact Lie algebras are precisely those whose Vogan diagrams have trivial automorphisms and {\it no} painted vertices (\cite[\S VI.10, discussion preceding table (6.101)]{knp}). In that case, we obtain a slight refinement of the Auerbach theorem (\cite[Theorem 6.82]{hm}):

\begin{corollary}\label{cor:auerb}
  Let $G$ be a compact, connected, semisimple Lie group. Then, the set of pairs that generate $G$ ${}_e$topologically is ${}_z$open and ${}_e$dense in $G^2$.
  \qedhere
\end{corollary}

Another consequence (recall the characteristic indices mentioned in \Cref{re:charind}): 

\begin{corollary}\label{cor:2cgen0}
  The characteristic index of a connected Lie group generated topologically by two circles can be arbitrarily large.  
\end{corollary}
\begin{proof}
  Take $G$ to be, say, a product of copies of $SL(2,\bR)$. The latter's Vogan diagram consists of a single black vertex, so \Cref{th:strreggen} applies in its present partially-proven state; the characteristic index is the number of $SL(2)$ factors, as $SL(2,\bR)$ is homeomorphic to $\bS^1\times \bR^2$.
\end{proof}

\pff{th:strreggen-cplx}{th:strreggen}{: the complex case.}
\begin{th:strreggen-cplx}
  We fix $H_1$ throughout and seek to prove that it generates together with some conjugate of $H_2$. Let
  \begin{equation}\label{eq:rtdec}
    \fg = \fh\oplus \bigoplus_{\alpha}\fg_{\alpha}
  \end{equation}
  be the root-space decomposition, for a Cartan subalgebra $\fh$ that contains $\fh_1:=Lie(H_1)$. Fix $e_{\alpha}\in \fg_{\alpha}$ and $h_{\alpha}\in \fh$ spanning copies of $\mathfrak{sl}_2$, as usual (e.g. \cite[\S 8.3, Proposition]{hum}):
  \begin{equation*}
    [h_{\alpha},e_{\alpha}] = 2e_{\alpha},\quad [h_{\alpha},e_{-\alpha}] = 2e_{-\alpha},\quad [e_{\alpha},e_{-\alpha}] = h_{\alpha}.
  \end{equation*}
  
  Arguing as in the proof of the previously-treated case, a generic conjugate of $\fh_1$ will have non-trivial components in all $\fg_{\alpha}$. By \Cref{le:tormod}\Cref{item:26} (and its proof) $\fg_{\alpha}$ is irreducible with multiplicity 2 in the decomposition \Cref{eq:rtdec} of $H_1$-modules, its mirror summand being $\fg_{-\alpha}$. This means that generically, the Lie algebra
  \begin{equation*}
    \fl:=Lie\overline{\langle H_1,H_2\rangle}
  \end{equation*}
  will contain vectors of the form $ae_{\alpha}+be_{-\alpha}$ with (complex) $a,b\ne 0$ and the 2-dimensional $H_1$-modules they respectively generate. Furthermore, arguing as in the proof of the previous case of the theorem, we have quite a bit of freedom in our choice of $a$ and $b$: \Cref{eq:diffad} means that for small real $s,t$, arbitrary complex $c,d$ and $h\in \fh$ (where we can assume $\fh_2$ starts out, before conjugating by some generic $g\in G$) we have
  \begin{equation}\label{eq:tay}
    Ad_{\exp sde_{-\alpha}} Ad_{\exp tce_{\alpha}} h = h - t\alpha(h) c e_{\alpha} + s \alpha(h) d e_{-\alpha}\quad +\quad \text{higher terms},
  \end{equation}
  meaning higher in total degree in $s$ and $t$. This means that we can find conjugates of $\fh_2$ containing $ae_{\alpha}+be_{-\alpha}$ for $a,b$ ranging over a {\it dense} subset of $\bC^2$.
  
  Now, since (the adjoint action of) an element of $\fh_1$ scales $e_{\alpha}$ by some purely imaginary $t\in \bR i$ and $e_{-\alpha}$ by $-t$, that 2-dimensional module is simply
  \begin{equation*}
    \{z a e_{\alpha} + \overline{z}b e_{-\alpha}\ |\ z\in \bC\}.
  \end{equation*}  
  The Lie algebra $\fl$ will then also contain the element
  \begin{equation}\label{eq:abih}
    -2abi h_{\alpha} = [ae_{\alpha}+be_{-\alpha}, aie_{\alpha}-bie_{-\alpha}].
  \end{equation}
  The above-noted freedom in producing the element $ae_{\alpha}+be_{-\alpha}$ allows us to assume, for instance, that $ab\ne 0$ is not real. But then the adjoint action of \Cref{eq:abih} will not preserve
  \begin{equation*}
    \mathrm{span}_{\bR}\{ae_{\alpha}+be_{-\alpha}, aie_{\alpha}-bie_{-\alpha}\},
  \end{equation*}
  which in turn means that $\fl$ contains {\it both} $\fg_{\pm\alpha}$. This finishes the proof, given that the root spaces $\fg_{\alpha}$ generate $\fg$ as a Lie algebra.
\end{th:strreggen-cplx}

\pff{th:strreggen-ntriv}{th:strreggen}{: Vogan diagrams with non-trivial involution.}
\begin{th:strreggen-ntriv}
  We fix some notation:
  \begin{itemize}
  \item $\fg$ for the Lie algebra of $G$ and $\fh_{i}$ for those of $H_i$ respectively, as before;
  \item $(\fg,\theta,\fk,\fp)$ for a Cartan decomposition of $\fg$ as in \Cref{not:quad};
  \item $\fh\subset \fg$ for a maximally compact $\theta$-invariant Cartan subalgebra, with its attendant decomposition
    \begin{equation*}
      \fh = \ft\oplus \fa:=(\fh\cap \fk)\oplus(\fh\cap \fp)
    \end{equation*}
    and the standing assumption that $\fh_1\le \ft$;
  \item `$\bC$' superscripts for complexifications throughout, as in $\fg^{\bC}$, etc;
  \item $\Delta$ for the root system attached to $(\fg^{\bC},\fh^{\bC})$, with its positive half $\Delta^{+}$ chosen so as to play well with all of this data, as explained in the discussion following \Cref{not:quad}. 
  \end{itemize}
  The only reason we cannot apply the argument in the trivial-involution case is precisely the fact that for simple roots $\alpha$ with $\theta\alpha\ne \alpha$ (i.e. members of size-2 orbits under $\theta$: the novel feature of the present case) is that for each such root we have an irreducible 2-dimensional (real) $H_1$-subrepresentation of $\fg$, appearing with multiplicity {\it two}. The device that will rid us of this difficulty is very similar in spirit to the proof of the complex case, albeit with some minor notational inconveniences. 

  Fix a simple root $\alpha\in \Delta^+$ with $\theta\alpha\ne \alpha$. According to \cite[Theorem 6.88]{knp} (and its proof) the summands $\fk$ and $\fp$ can be re-assembled into a {\it compact} real form
  \begin{equation*}
    \fu:=\fk\oplus i\fp
  \end{equation*}
  of the complexification $\fg^{\bC}$. That compact form further decomposes as
  \begin{equation*}
    \fu = \sum_{\alpha\in \Delta}\bR ih_{\alpha} + \sum_{\alpha\in \Delta}\bR(e_{\alpha}-e_{-\alpha}) +  \sum_{\alpha\in \Delta}\bR i(e_{\alpha}+e_{-\alpha}),
  \end{equation*}
  where
  \begin{itemize}
  \item $h_\alpha$ and $e_{\pm\alpha}$ are elements of $\fg^{\bC}$ as in the preceding proof of the complex case: our $h_{\alpha}$ is generally scaled by a non-zero real scalar as compared to the $H_{\alpha}$ of \cite[Chapter VI]{knp} (e.g. \cite[\S IV.6]{knp}), but that scaling is irrelevant.
  \item $\theta$ is the extension of the original involution to $\fg^{\bC}$, it preserves $\fu$, and sends the simple-root vector $e_{\alpha}$ to
    \begin{equation*}
      \begin{cases}
        e_{\alpha}&\text{ if }\theta\alpha=\alpha\text{ and $\alpha$ is not painted};\\
        -e_{\alpha}&\text{ if }\theta\alpha=\alpha\text{ and $\alpha$ {\it is} painted};\\
        e_{\theta\alpha}&\text{ otherwise}.\\
      \end{cases}
    \end{equation*}
  \end{itemize}
  The two isomorphic 2-dimensional irreducible $H_1$-subrepresentations of $\fg$ corresponding to the roots $\alpha$ and $\theta\alpha$ will then be
  \begin{equation}\label{eq:vw}
    \mathrm{span}_{\bR}\{v_{\alpha+}, w_{\alpha+}\}\subset \fk
  \end{equation}
  and 
  \begin{equation}\label{eq:ivw}
    \mathrm{span}_{\bR}\{iv_{\alpha-}, iw_{\alpha-}\}\subset \fp,
  \end{equation}
  where
  \begin{align*}
    v_{\alpha}&:=e_{\alpha}-e_{-\alpha}\\
    w_{\alpha}&:=i(e_{\alpha}+e_{-\alpha})\text{ and }\\
    \bullet_{\alpha+}&:=\bullet_{\alpha}+\bullet_{\theta\alpha}\\
    \bullet_{\alpha-}&:=\bullet_{\alpha}-\bullet_{\theta\alpha}\\
  \end{align*}
  for $\bullet$ either `$v$' or `$w$'. Note that $[v_{\alpha},w_{\alpha}]=2ih_{\alpha}$ up to irrelevant scaling (e.g. \cite[proof of Theorem 6.11, p.354]{knp}).

  An isomorphism that will identify \Cref{eq:vw} and \Cref{eq:ivw} as real $H_1$-representations is given by
  \begin{equation*}
    v_{\alpha+}\mapsto iw_{\alpha-},\quad w_{\alpha+}\mapsto -iv_{\alpha-}.
  \end{equation*}
  Now, by the same reasoning employed twice already, via \Cref{eq:diffad}, the Lie algebra $\fl$ of a Lie group generated by $H_1$ and a generic conjugate of $H_2$ will contain
  \begin{equation*}
    \mathrm{span}_{\bR}\{a v_{\alpha+} + biw_{\alpha-},\ a w_{\alpha+} - biv_{\alpha-}\}
  \end{equation*}
  for non-zero real $a,b$. All of these spaces are contained in
  \begin{equation}\label{eq:pqrs}
    \{pv_{\alpha+}+q i w_{\alpha-}+r w_{\alpha+} + s (-iv_{\alpha-})\ |\ p,q,r,s\in \bR\text{ and }ps=qr\}.
  \end{equation}
  On the other hand, arguing as in the complex-case proof, i.e. by conjugating elements of $\ft$ with exponentials of elements in \Cref{eq:vw} and \Cref{eq:ivw} as in \Cref{eq:tay}, we can ensure that the component of some conjugate $Ad_g\fh_2$ in the direct sum of \Cref{eq:vw} and \Cref{eq:ivw} is {\it not} contained in \Cref{eq:pqrs}. As in the complex case, this will then imply that
  \begin{equation*}
    \fl=Lie\overline{\langle H_1, Ad_g H_2\rangle}
  \end{equation*}
  contains both \Cref{eq:vw} and \Cref{eq:ivw}. Extended to all simple roots in $\Delta^+$, this proves the claim: $\fl$ will contain all non-trivial $H_1$-submodules of $\fg$, which generate $\fg$ as a Lie algebra.
\end{th:strreggen-ntriv}



\begin{thebibliography}{10}

\bibitem{ad-lie}
J.~Frank Adams.
\newblock {\em Lectures on {L}ie groups}.
\newblock W. A. Benjamin, Inc., New York-Amsterdam, 1969.

\bibitem{aub3}
H.~Auerbach.
\newblock Sur les groupes linéaires bornés (iii).
\newblock {\em Studia Mathematica}, 5(1):43--49, 1934.

\bibitem{bol-bk}
B\'{e}la Bollob\'{a}s.
\newblock {\em Combinatorics}.
\newblock Cambridge University Press, Cambridge, 1986.
\newblock Set systems, hypergraphs, families of vectors and combinatorial
  probability.

\bibitem{bor}
Armand Borel.
\newblock {\em Linear algebraic groups}, volume 126 of {\em Graduate Texts in
  Mathematics}.
\newblock Springer-Verlag, New York, second edition, 1991.

\bibitem{cas-alg}
B.~Casselman.
\newblock Compact groups as algebraic groups, 2015.
\newblock Available at
  \url{https://www.math.ubc.ca/~cass/research/pdf/Compact-algebraic.pdf}.

\bibitem{chv}
Claude Chevalley.
\newblock {\em Theory of {L}ie groups. {I}}, volume~8 of {\em Princeton
  Mathematical Series}.
\newblock Princeton University Press, Princeton, NJ, 1999.
\newblock Fifteenth printing, Princeton Landmarks in Mathematics.

\bibitem{col-j}
Michael~J. Collins.
\newblock On {J}ordan's theorem for complex linear groups.
\newblock {\em J. Group Theory}, 10(4):411--423, 2007.

\bibitem{cr}
Charles~W. Curtis and Irving Reiner.
\newblock {\em Representation theory of finite groups and associative
  algebras}.
\newblock Pure and Applied Mathematics, Vol. XI. Interscience Publishers, a
  division of John Wiley \& Sons, New York-London, 1962.

\bibitem{dgk-helly}
Ludwig Danzer, Branko Gr\"{u}nbaum, and Victor Klee.
\newblock Helly's theorem and its relatives.
\newblock In {\em Proc. {S}ympos. {P}ure {M}ath., {V}ol. {VII}}, pages
  101--180. Amer. Math. Soc., Providence, R.I., 1963.

\bibitem{fh}
William Fulton and Joe Harris.
\newblock {\em Representation theory}, volume 129 of {\em Graduate Texts in
  Mathematics}.
\newblock Springer-Verlag, New York, 1991.
\newblock A first course, Readings in Mathematics.

\bibitem{har-ag}
Joe Harris.
\newblock {\em Algebraic geometry}, volume 133 of {\em Graduate Texts in
  Mathematics}.
\newblock Springer-Verlag, New York, 1992.
\newblock A first course.

\bibitem{helg}
Sigurdur Helgason.
\newblock {\em Differential geometry, {L}ie groups, and symmetric spaces},
  volume~34 of {\em Graduate Studies in Mathematics}.
\newblock American Mathematical Society, Providence, RI, 2001.
\newblock Corrected reprint of the 1978 original.

\bibitem{hm}
Karl~H. Hofmann and Sidney~A. Morris.
\newblock {\em The structure of compact groups}, volume~25 of {\em De Gruyter
  Studies in Mathematics}.
\newblock De Gruyter, Berlin, 2013.
\newblock A primer for the student---a handbook for the expert, Third edition,
  revised and augmented.

\bibitem{hum}
James~E. Humphreys.
\newblock {\em Introduction to {L}ie algebras and representation theory}.
\newblock Springer-Verlag, New York-Berlin, 1972.
\newblock Graduate Texts in Mathematics, Vol. 9.

\bibitem{iw}
Kenkichi Iwasawa.
\newblock On some types of topological groups.
\newblock {\em Ann. of Math. (2)}, 50:507--558, 1949.

\bibitem{knp}
Anthony~W. Knapp.
\newblock {\em Lie groups beyond an introduction}, volume 140 of {\em Progress
  in Mathematics}.
\newblock Birkh\"{a}user Boston, Inc., Boston, MA, second edition, 2002.

\bibitem{mz}
Deane Montgomery and Leo Zippin.
\newblock {\em Topological transformation groups}.
\newblock Interscience Publishers, New York-London, 1955.

\bibitem{ragh}
M.~S. Raghunathan.
\newblock {\em Discrete subgroups of {L}ie groups}.
\newblock Springer-Verlag, New York-Heidelberg, 1972.
\newblock Ergebnisse der Mathematik und ihrer Grenzgebiete, Band 68.

\bibitem{rob}
Alain Robert.
\newblock {\em Introduction to the representation theory of compact and locally
  compact groups}, volume~80 of {\em London Mathematical Society Lecture Note
  Series}.
\newblock Cambridge University Press, Cambridge-New York, 1983.

\bibitem{su}
J.~Schreier and S.~Ulam.
\newblock Sur le nombre des générateurs d'un groupe topologique compact et
  connexe.
\newblock {\em Fundamenta Mathematicae}, 24(1):302--304, 1935.

\bibitem{ser-lie}
Jean-Pierre Serre.
\newblock {\em Lie algebras and {L}ie groups}, volume 1500 of {\em Lecture
  Notes in Mathematics}.
\newblock Springer-Verlag, Berlin, 2006.
\newblock 1964 lectures given at Harvard University, Corrected fifth printing
  of the second (1992) edition.

\end{thebibliography}
\def\polhk#1{\setbox0=\hbox{#1}{\ooalign{\hidewidth
  \lower1.5ex\hbox{`}\hidewidth\crcr\unhbox0}}}

\addcontentsline{toc}{section}{References}

\Addresses

\end{document}